\documentclass{mfat}
\pagespan{132}{141}

\theoremstyle{plain}
\newtheorem{thm}{Theorem}
\newtheorem{lem}{Lemma}

\theoremstyle{definition}

\theoremstyle{remark}

\begin{document}

\title[A stochastic integral of operator-valued functions]
      {A stochastic integral of operator-valued functions}

\author{Volodymyr Tesko}

\address{Institute of Mathematics, National Academy of
         Sciences of Ukraine, 3 Teresh\-chenkivs'ka, Kyiv, 01601,
         Ukraine}

\email{tesko@imath.kiev.ua}





\date{16/01/2008}

\dedicatory{To Professor M. L. Gorbachuk on the occasion of his
70th birthday.}



\begin{abstract}
      In this note we define and study a Hilbert space-valued stochastic integral of
      operator-valued functions with respect to Hilbert space-valued
      measures.
      We show that this integral generalizes the classical  It\^{o} stochastic integral
      of adapted processes with respect to normal martingales
      and the It\^{o} integral in a Fock space.
\end{abstract}

\maketitle


\section{Introduction}

Here  and subsequently, we fix a real number $T>0$. Let
$\mathcal{H}$ be a  complex Hilbert space, $M$ be a fixed vector
from $\mathcal{H}$ and $
  [0,T]\ni t\mapsto E_t
$ be a resolution of identity in $\mathcal{H}$. Consider the
$\mathcal{H}$-valued function  ({\it abstract martingale})
$$
   [0,T]\ni t\mapsto M_t:=E_tM\in\mathcal{H}.
$$
In this paper we construct and study an integral
\begin{equation}\label{eq.1.1}
      \int_{[0,T]}A(t)\,dM_t
\end{equation}
for a certain class of operator-valued functions $[0,T]\ni t\mapsto
A(t)$ whose values are linear operators in the space $\mathcal{H}$.
We define such an integral as an element of the Hilbert space
$\mathcal{H}$ and call it a {\it Hilbert space-valued stochastic
  integral} (or {\it $H$-stochastic integral}).
By analogy with the classical
integration theory we  first define integral (\ref{eq.1.1}) for a
certain class of simple operator-valued functions and then extend
this definition to a  wider class.

We illustrate our abstract constructions with a few examples. Thus, 
we show that the classical It\^{o} stochastic integral is a
particular case of the $H$-stochastic integral. Namely, let
$\mathcal{H}:=L^2(\Omega,\mathcal{A},P)$ be a space of square
integrable functions on a complete probability space
$(\Omega,\mathcal{A},P)$, $\{\mathcal{A}_t\}_{t\in[0,T]}$ be a
filtration satisfying the usual conditions and $\{N_t\}_{t\in[0,T]}$
be a normal martingale on $(\Omega,\mathcal{A},P)$ with respect to
$\{\mathcal{A}_t\}_{t\in[0,T]}$, i.e.,
$$
   \{N_t\}_{t\in[0,T]}\quad\text{and}\quad
   \{N_t^2-t\}_{t\in[0,T]}
$$
are martingales for $\{\mathcal{A}_t\}_{t\in[0,T]}$. It follows from
the properties of martingales that
\begin{equation*}
   N_t=\mathbb{E}[N_T|\mathcal{A}_t],\quad t\in[0,T],
\end{equation*}
where $\mathbb{E}[\,\cdot\,|\mathcal{A}_t]$ is a conditional
expectation with respect to the $\sigma$-algebra $\mathcal{A}_t$. It
is well known that $\mathbb{E}[\,\cdot\,|\mathcal{A}_t]$ is the
orthogonal projector in the space $L^2(\Omega,\mathcal{A},P)$ onto
its subspace $L^2(\Omega,\mathcal{A}_t,P)$ and, moreover, the
corresponding projector-valued function $
   \mathbb{R}_+\ni t\mapsto E_t:=\mathbb{E}[\,\cdot\,|\mathcal{A}_t]
$ is a resolution of identity in $L^2(\Omega,\mathcal{A},P)$, see
e.g. \cite{S, BZU87, BZU89, M, B98a}. In this way the normal
martingale $\{N_t\}_{t\in[0,T]}$ can be interpreted as an abstract
martingale, i.e.,
$$
   [0,T]\ni t\mapsto N_t=\mathbb{E}[N_T|\mathcal{A}_t]=E_tN_T\in\mathcal{H}.
$$
Hence, in the space $L^2(\Omega,\mathcal{A},P)$ we can construct the
$H$-stochastic integral with respect to  the normal martingale
$N_t$. Let $F\in L^2([0,T]\times\Omega,dt\times P)$ be a square
integrable stochastic process adapted to the filtration
$\{\mathcal{A}_t\}_{t\in[0,T]}$. We consider the operator-valued
function $[0,T]\ni t\mapsto A_F(t)$ whose values are operators
$A_F(t)$ of multiplication by the function $F(t)=F(t,\cdot)\in
L^2(\Omega,\mathcal{A},P)$ in the space $L^2(\Omega,\mathcal{A},P)$,
$$
   L^2(\Omega,\mathcal{A},P)\supset\mathop{\rm Dom}(A_F(t))\ni
   G\mapsto A_F(t)G:=F(t)G\in L^2(\Omega,\mathcal{A},P).
$$
In this paper we prove that the $H$-stochastic integral of $[0,T]\ni
t\mapsto A_F(t)$ coincides with the classical It\^{o} stochastic
integral $\int_{[0,T]}F(t)\,dN_t$ of $F$. That is,
$$
   \int_{[0,T]}A_F(t)\,dN_t=\int_{[0,T]}F(t)\,dN_t.
$$

In the last part of this note we show that the It\^{o} integral in a
Fock space is the $H$-stochastic integral and establish a connection
of such an integral with the classical It\^{o} stochastic integral.
The corresponding results are given without proofs (the proofs
will be given in a forthcoming publication). Note that the It\^{o}
integral in a Fock space is a useful tool in the quantum
stochastic calculus, see e.g. \cite{Attal} for more details.

We remark that in \cite{BZU87, BZU89} the authors gave a
definition  of the operator-valued stochastic integral
\begin{equation*}
      B:=\int_{[0,T]}A(t)\,dE_t
\end{equation*}
for a family $\{A(t)\}_{t\in[0,T]}$ of commuting normal operators in
$\mathcal{H}$. Such an integral was defined using a spectral theory of
commuting normal operators. It is clear that for a fixed vector
$M\in \mathop{\rm Dom}(B)\subset \mathcal{H}$ the formula
\begin{equation*}
      \int_{[0,T]}A(t)\,dM_t:=\Big(\int_{[0,T]}A(t)\,dE_t\Big)M
\end{equation*}
can be regarded as a definition of integral (\ref{eq.1.1}). In this
way we obtain another definition of integral (\ref{eq.1.1})
different from the one we have proposed in this paper.



\section{The construction of the $H$-stochastic integral}
\enlargethispage{1cm}

Let $\mathcal{H}$ be a complex Hilbert space,
$\mathcal{L}(\mathcal{H})$ be a space of all bounded linear
operators in $\mathcal{H}$, $M\neq 0$ be a fixed vector from
$\mathcal{H}$ and
$$
  [0,T]\ni t\mapsto E_t\in\mathcal{L}(\mathcal{H})
$$
be a resolution of identity in $\mathcal{H}$, that is a
right-continuous increasing family of orthogonal projections in
$\mathcal{H}$ such that $E_T=1$.
Note that the resolution of identity $E$ can be regarded as a
projector-valued measure $\mathcal{B}([0,T])\ni \alpha\mapsto
E(\alpha)\in\mathcal{L}(\mathcal{H})$ on the Borel $\sigma$-algebra
$\mathcal{B}([0,T])$. Namely, for any interval $(s,t]\subset [0,T]$
we set
$$
   E((s,t]):=E_t-E_s,\quad E(\{0\})
   :=E_0,\quad E(\varnothing):=0,
$$
and extend this definition to all Borel subsets of $[0,T]$, see e.g.
\cite{BSU90}  for more details.

By definition, the $\mathcal{H}$-valued function
\begin{equation*}
      [0,T]\ni t\mapsto M_t:=E_tM\in \mathcal{H}
\end{equation*}
is an {\it abstract martingale} in the Hilbert space $\mathcal{H}$.

In this section we give a definition of integral (\ref{eq.1.1}) for
a certain class of operator-valued functions with respect to the
abstract martingale $M_t$. A construction of such an integral is
given step-by-step, beginning with the simplest class of
operator-valued functions. Let us introduce the required class of
simple functions.

For each point $t\in[0,T]$, we denote by
$$
   \mathcal{H}_{M}(t):=\mathop{\rm span}\{M_{s_2}-M_{s_1}
   \,|\,(s_1,s_2]\subset(t,T]\}\subset
   \mathcal{H}
$$
the linear span of the set $\{M_{s_2}-M_{s_1}
   \,|\,(s_1,s_2]\subset(t,T]\}$ in $\mathcal{H}$ and
by
$$
   \mathcal{L}_{M}(t)=\mathcal{L}(\mathcal{H}_{M}(t)\to\mathcal{H})
$$
the set of all linear operators in $\mathcal{H}$  that
continuously act from $\mathcal{H}_{M}(t)$ to $\mathcal{H}$. The
increasing family
$\mathcal{L}_{M}=\{\mathcal{L}_{M}(t)\}_{t\in[0,T]}$ will play here
a role of the filtration $\{\mathcal{A}_t\}_{t \in [0,T]}$ in the
classical martingale theory.

For a fixed $t\in[0,T)$, a linear operator $A$ in  $ \mathcal{H}$ 
will be called {\it $\mathcal{L}_{M}(t)$-measurable} if
\begin{itemize}
  \item[(i)] $A\in\mathcal{L}_{M}(t)$  and, for all $s\in[t,T)$,
       \begin{equation*}
          \|A\|_{\mathcal{L}_{M}(t)}=\|A\|_{\mathcal{L}_{M}(s)}
          :=\sup
          \Big\{\frac{\|Ag\|_{\mathcal{H}}}{\|g\|_{\mathcal{H}}}\,\Big|\,
          g\in\mathcal{H}_{M}(s),\,g\neq 0\Big\}.
       \end{equation*}
  \item[(ii)] $A$ is partially commuting with the resolution of identity $E$. More precisely,
      \begin{equation*}
         AE_sg=E_sAg,\quad g\in\mathcal{H}_{M}(t),\quad s\in[t,T].
      \end{equation*}
\end{itemize}

Such a definition of $\mathcal{L}_{M}(t)$-measurability is motivated
by a number of reasons:
\begin{itemize}
  \item $\mathcal{L}_{M}(t)$-measurability is a natural generalization of the usual
        $\mathcal{A}_t$-measurability in classical stochastic calculus,
        see Lemma \ref{l.1} (Section \ref{s.3}) for more details;
      \item in some sense, $\mathcal{L}_{M}(t)$-measurability (for
        each $t$) is the minimal restriction on the behavior of a
        simple operator-valued function $[0,T]\ni t\mapsto A(t)$
        that will allow us to obtain an analogue of the It\^{o}
        isometry property (see inequality (\ref{eq.3.3}) below) and
        to extend the $H$-stochastic integral from a simple class of
        functions to a wider one.
\end{itemize}

In what follows, it is convenient for us to call
$\mathcal{L}_{M}(T)$-measurable all linear operators in
$\mathcal{H}$. Evidently, if a linear operator $A$ in $\mathcal{H}$
is $\mathcal{L}_{M}(t)$-measurable for some $t\in[0,T]$ then $A$ is
$\mathcal{L}_{M}(s)$-measurable for all $s\in [t,T]$.

\enlargethispage{1cm}
A family $\{A(t)\}_{t\in[0,T]}$ of linear operators  in
$\mathcal{H}$  will be called {\it a simple $\mathcal{L}_M$-adapted
operator-valued function on} $[0,T]$ if, for each $t\in[0,T]$, the
operator $A(t)$ is $\mathcal{L}_{M}(t)$-measurable and there exists
a partition $0=t_0<t_1<\cdots<t_n=T$ of $[0,T]$ such that
\begin{equation}\label{eq.3.1}
      A(t)=\sum_{k=0}^{n-1}A_k\varkappa_{(t_k,t_{k+1}]}(t),\quad t\in[0,T],
\end{equation}
where $\varkappa_{\alpha}(\cdot)$ is the characteristic  function of
the Borel set $\alpha\in\mathcal{B}([0,T])$.

Let $S=S(M)$ denote the space of all simple $\mathcal{L}_M$-adapted
operator-valued functions on $[0,T]$. For a function $A\in S$ with
representation (\ref{eq.3.1}) we define an {\it $H$-stochastic
integral} of $A$ with respect to
the abstract martingale $M_t$ 
through the formula
\begin{equation}\label{eq.3.21}
   \int_{[0,T]}A(t)\,dM_t:=\sum_{k=0}^{n-1}A_k(M_{t_{k+1}}-M_{t_{k}})
   \in\mathcal{H}.
\end{equation}
We can show that this definition  does not depend on  the choice of
representation of the simple function $A$ in the space $S$.

In the space $S$ we introduce a quasinorm by setting
\begin{equation*}
      \|A\|_{S_2}:=\Big(\int_{[0,T]}\|A(t)\|_{\mathcal{L}_{M}(t)}^{2}\,d\mu(t)
      \Big)^{\frac{1}{2}}
      :=\Big(\sum_{k=0}^{n-1}\|A_k\|_{\mathcal{L}_{M}(t_k)}^{2}\mu((t_k,t_{k+1}])
      \Big)^{\frac{1}{2}}
\end{equation*}
for each $A\in S$ with representation (\ref{eq.3.1}). Here the
measure $\mu$ is defined by the formula
$$
   \mathcal{B}([0,T])\ni\alpha\mapsto
   \mu(\alpha):=\|M(\alpha)\|_{\mathcal{H}}^{2}=(E(\alpha)M,M)_{\mathcal{H}}\in\mathbb{R}_+,
$$
where $M(\alpha):=E(\alpha)M$ for all $\alpha\in\mathcal{B}([0,T])$,
in particular,
$$
   M((t_k,t_{k+1}]):=E((t_k,t_{k+1}])M=M_{t_{k+1}}-M_{t_{k}},
   \quad (t_k,t_{k+1}]\subset[0,T].
$$

The following statement is fundamental.

\begin{thm}\label{t.1}
       Let $A,B\in S$ and $a,b\in\mathbb{C}$. Then
       \begin{equation*}
          \int_{[0,T]}\Big(aA(t)+bB(t)\Big)\,dM_t
          =a\int_{[0,T]}A(t)\,dM_t+b\int_{[0,T]}B(t)\,dM_t
       \end{equation*}
       and
       \begin{equation}\label{eq.3.3}
             \Big\|\int_{[0,T]}A(t)\,dM_t\Big\|_{\mathcal{H}}^{2}
             \leq \int_{[0,T]}\|A(t)\|_{\mathcal{L}_{M}(t)}^{2}\,d\mu(t).
       \end{equation}
\end{thm}

\begin{proof}
      The first assertion is trivial.

      Let us check inequality  (\ref{eq.3.3}). Using (i), (ii)
      and properties of the resolution of identity $E$,
      for $A\in S$ with representation (\ref{eq.3.1}),
      we obtain
      \begin{equation*}
            \begin{split}
                   \Big\|\int_{[0,T]}A(t)\,dM_t\Big\|_{\mathcal{H}}^{2}
                 &
                   =\Big(\int_{[0,T]}A(t)\,dM_t,
                   \int_{[0,T]}A(t)\,dM_t\Big)_{\mathcal{H}}\\
                 &
                   =\sum_{k,m=0}^{n-1}\big(A_kM(\Delta_k),A_mM(\Delta_m)\big)_{\mathcal{H}}\\
                 &
                   =\sum_{k,m=0}^{n-1}\big(A_kE(\Delta_k)M,A_mE(\Delta_m)M\big)_{\mathcal{H}}\\
                 &
                   =\sum_{k,m=0}^{n-1}\big(E(\Delta_k)A_kE(\Delta_k)M,
                   E(\Delta_m)A_mE(\Delta_m)M\big)_{\mathcal{H}}\\
                 &
                   =\sum_{k=0}^{n-1}\big(A_kE(\Delta_k)M,A_kE(\Delta_k)M\big)_{\mathcal{H}}
                   =\sum_{k=0}^{n-1}\|A_kM(\Delta_k)\|_{\mathcal{H}}^{2}\\
                 &
                   \quad\leq\sum_{k=0}^{n-1}\|A_k\|_{\mathcal{L}_M(t_k)}^{2}\|
                   M(\Delta_k)\|_{\mathcal{H}}^{2}
                   =\sum_{k=0}^{n-1}\|A_k\|_{\mathcal{L}_{M}(t_k)}^{2}\mu(\Delta_k)\\
                 &
                   =\int_{[0,T]}\|A(t)\|_{\mathcal{L}_{M}(t)}^{2}\,d\mu(t),
            \end{split}
      \end{equation*}
      where $\Delta_k:=(t_k,t_{k+1}]$ for all
      $k\in\{0,\ldots,n-1\}$.
\end{proof}

\enlargethispage{1cm}
Inequality (\ref{eq.3.3}) enables us to extend the $H$-stochastic
integral to operator-valued functions $[0,T]\ni t\mapsto A(t)$
which are not necessarily simple. Namely, denote by $S_2=S_2(M)$ a
Banach space associated with the quasinorm $\|\cdot\|_{S_2}$. For
its construction, it is first necessary to pass from $S$ to the
factor space $$
   \dot{S}:=S/\{A\in S\,|\,\|A\|_{S_2}=0\}
$$
and then to take
the completion of $\dot{S}$. 
It is not difficult to see that elements  of the space $S_2$
are equivalence classes of operator-valued functions on $[0,T]$
whose values are linear operators in the space $\mathcal{H}$.

An operator-valued function $[0,T]\ni t\mapsto A(t)$ will be called
{\it $H$-stochastic integrable} with respect to $M_t$ if $A$ belongs
to the space $S_2$. It follows from the definition of the space
$S_2$ that for each $A\in S_2$ there exists a sequence
$(A_n)_{n=0}^{\infty}$ of simple operator-valued functions $A_n\in
S$ such that
\begin{equation}\label{eq.3.4}
              \int_{[0,T]}\|A(t)-A_n(t)\|_{\mathcal{L}_M(t)}^{2}\,d\mu(t)\rightarrow 0
              \quad {\text{\rm as}}\quad n\rightarrow\infty.
\end{equation}
Due to (\ref{eq.3.3}), for such a sequence $(A_n)_{n=0}^{\infty}$,
the limit
$$
   \lim_{n\to\infty}\int_{[0,T]}A_n(t)\,dM_t
$$
exists in $\mathcal{H}$ and does not dependent on the  choice of the
sequence $(A_n)_{n=0}^{\infty}\subset S$ satisfying (\ref{eq.3.4}).
We denote this limit by
$$
   \int_{[0,T]}A(t)\,dM_t:=\lim_{n\to\infty}\int_{[0,T]}A_n(t)\,dM_t
$$
and call it the {\it $H$-stochastic integral of} $A\in S_2$ with
respect to the abstract martingale $M_t$. It is clear that for all
$A\in S_2$ the assertions of Theorem~1 still hold.

Note one simple property of the integral introduced above. Let $U$
be some unitary operator acting from $\mathcal{H}$ onto another
complex Hilbert space $\mathcal{K}$. Then
\begin{equation*}
      [0,T]\ni t\mapsto G_t:=UM_t\in \mathcal{K}
\end{equation*}
is an abstract martingale in the space $\mathcal{K}$ because, for
any $t\in[0,T]$,
\begin{equation*}
      G_t=UM_t=X_tG,\quad X_t:=UE_tU^{-1},\quad
      G:=UM\in\mathcal{K},
\end{equation*}
and $X_t$ is a resolution of identity in the space $\mathcal{K}$.

Let an operator-valued function $[0,T]\ni t\mapsto A(t)$  be
$H$-stochastic integrable with respect to $M_t$. One can show that
the operator-valued function $[0,T]\ni t\mapsto UA(t)U^{-1}$ is
$H$-stochastic integrable with respect to $G_t$ and
\begin{equation*}
      U\Big(\int_{[0,T]}A(t)\,dM_t\Big)
      =\int_{[0,T]}UA(t)U^{-1}\,dG_t\in\mathcal{K}.
\end{equation*}



\section{The It\^{o} stochastic integral as
        an $H$-stochastic integral}\label{s.3}

Let
$
   (\Omega, \mathcal{A}, P)
$ be a complete probability space and $\{\mathcal{A}_t\}_{t \in
[0,T]}$  be a right continuous filtration.   Suppose that the
$\sigma$-algebra $\mathcal{A}_0$ contains all  $P$-null sets of
$\mathcal{A}$ and $\mathcal{A}=\mathcal{A}_T$. Moreover, we assume
that $\mathcal{A}_0$ is trivial, i.e., every set
$\alpha\in\mathcal{A}_0$ has probability $0$ or $1$.

Let $N=\{N_t\}_{t\in[0,T]}$ be a {\it normal martingale}  on
$(\Omega, \mathcal{A}, P)$ with respect to  $\{\mathcal{A}_t\}_{t
\in [0,T]}$. That is,  $N_t\in L^2(\Omega, \mathcal{A}_t, P)$ for
all $t\in[0,T]$ and
\begin{equation*}
   \mathbb{E}[N_t-N_s|\mathcal{A}_s]=0,\quad
   \mathbb{E}[(N_t-N_s)^2|\mathcal{A}_s]=t-s
\end{equation*}
for all $s,t\in[0,T]$ such that $s<t$. Without  loss of generality
one can assume that $N_0=0$. Note that there are many examples of
normal martingales, --- the Brownian motion, the compensated Poisson
process, the Az${\rm \acute{e}}$ma martingales and others, see for
instance  \cite{Emery89, DMM92, M}.

We will denote by $L_a^2([0,T]\times\Omega)$ the set of all
functions (equivalence classes), adapted to the filtration
$\{\mathcal{A}_t\}_{t \in [0,T]}$, from the space
$$
   L^2([0,T]\times\Omega)
   :=L^2([0,T]\times\Omega,\mathcal{B}([0,T])\times\mathcal{A},dt\times P)
$$
where $dt$ is the Lebesgue measure on $\mathcal{B}([0,T])$.

Let us show that the {\it It\^{o} stochastic integral}
$\int_{[0,T]} F(t)\,dN_t$ of $F\in L_a^2([0,T]\times\Omega)$ with
respect to the normal martingale $N$ can be considered as an
$H$-stochastic integral (see e.g. \cite{GSk75, P90} for the
definition and properties of the classical It\^{o}  integral). To
this end, we set
$
   \mathcal{H}:=L^2(\Omega,\mathcal{A},P)
$ and consider, in this space, the resolution of identity
$$
   [0,T]\ni t\mapsto
   E_t:=\mathbb{E}[\,\cdot\,|\mathcal{A}_t]\in\mathcal{L}(\mathcal{H})
$$ generated by the filtration $\{\mathcal{A}_t\}_{t \in [0,T]}$.
Let $M:=N_T\in L^2(\Omega,\mathcal{A},P)$, then the
corres\-ponding abstract martingale
\begin{equation*}
      [0,T]\ni t\mapsto N_t:=E_tN_T=
      \mathbb{E}[N_T|\mathcal{A}_t]\in \mathcal{H}
\end{equation*}
is our normal martingale.  Note also that
$$
   \mu([0,t])=\|N([0,t])\|_{L^2(\Omega,\mathcal{A},P)}^{2}=
   \|N_t\|_{L^2(\Omega,\mathcal{A},P)}^{2}=\mathbb{E}[N_t^2]
   =\mathbb{E}[N_t^2\,|\,\mathcal{A}_0]=t,
$$
i.e., $\mu$ is the Lebesgue measure on $\mathcal{B}([0,T])$.

In the context of this section,
$\mathcal{L}_{M}(t)$-measurabi\-lity is equivalent to the usual
$\mathcal{A}_t$-measurabi\-lity. More precisely, the following
result holds.

\begin{lem}\label{l.1}
       Let $t\in[0,T)$. For given $F\in L^2(\Omega,\mathcal{A},P)$
       the operator $A_F$ of  multiplication by the function $F$ in
       the space $L^2(\Omega,\mathcal{A},P)$
       is $\mathcal{L}_{N}(t)$-measurable 
       if and only if the function $F$ is
       $\mathcal{A}_t$-measurable, i.e., $F=\mathbb{E}[F|\mathcal{A}_t]$.
       Moreover, if $F\in L^2(\Omega,\mathcal{A},P)$ is an $\mathcal{A}_t$-measurable function then
       \begin{equation}\label{eq.3.0.101}
         \|A_F\|_{\mathcal{L}_{N}(t)}=\|A_F\|_{\mathcal{L}_{N}(s)}
         =\|F\|_{L^2(\Omega,\mathcal{A},P)},
         \quad s\in[t,T).
       \end{equation}
\end{lem}

\begin{proof}
      Suppose $F\in L^2(\Omega,\mathcal{A},P)$ is an $\mathcal{A}_t$-measurable function.
      Let us show that the operator $A_F$  is $\mathcal{L}_{N}(t)$-measurable.

      First, we prove that $A_F\in\mathcal{L}_N(t)$. Taking into
      account that $F$ is an $\mathcal{A}_t$-measurable function,
      $\{N_t\}_{t\in[0,T]}$ is the normal martingale and the
      $\sigma$-algebra $\mathcal{A}_0$ is trivial, for any interval
      $(s_1,s_2]\subset(t,T]$, we obtain
\begin{equation*}
            \begin{split}
                  \|A_F(N_{s_2}-N_{s_1})\|
                  _{L^2(\Omega,\mathcal{A},P)}^{2}
                &
                  =\|F(N_{s_2}-N_{s_1})\|_{L^2(\Omega,\mathcal{A},P)}^{2}
                  =\mathbb{E}[F^2(N_{s_2}-N_{s_1})^2]\\
                &
                  =\mathbb{E}[F^2(N_{s_2}-N_{s_1})^2|\mathcal{A}_{0}]
                  =\mathbb{E}\big[F^2\mathbb{E}[(N_{s_2}-N_{s_1})^2|
                  \mathcal{A}_{s_1}]\big|\mathcal{A}_{0}\big]\\
                &
                  =\mathbb{E}\big[F^2\mathbb{E}[(N_{s_2}-N_{s_1})^2|\mathcal{A}_{s_1}]\big]
                  =\mathbb{E}[F^2](s_2-s_1)\\
                &
                  =\mathbb{E}[F^2]\mathbb{E}[(N_{s_2}-N_{s_1})^2]\\
&
                  =\|F\|_{L^2(\Omega,\mathcal{A},P)}^{2}
                  \|N_{s_2}-N_{s_1}\|_{L^2(\Omega,\mathcal{A},P)}^{2}.
            \end{split}
\end{equation*}

      We can similarly show that
      \begin{equation*}
         \|A_FG\|_{L^2(\Omega,\mathcal{A},P)}^{2}
         =\|F\|_{L^2(\Omega,\mathcal{A},P)}^{2}
         \|G\|_{L^2(\Omega,\mathcal{A},P)}^{2}
      \end{equation*}
      for all $G\in \mathcal{H}_{N}(t)={\rm span}\{N_{s_2}-N_{s_1}\,|(s_1,s_2]\subset(t,T]\}$.
      Hence $A_F\in\mathcal{L}_N(t)$ and, moreover, equality (\ref{eq.3.0.101}) takes place.

      Let us check that $A_F$ is partially commuting with $E$, i.e.,
      \begin{equation*}
         A_FE_sG=E_sA_FG,\quad G\in\mathcal{H}_{N}(t),\quad s\in[t,T].
      \end{equation*}
      Since $F\in L^2(\Omega,\mathcal{A},P)$ is an $\mathcal{A}_t$-measurable function and
      $FG\in L^2(\Omega,\mathcal{A},P)$,
      for any  $s\in[t,T]$ and any function
      $G\in\mathcal{H}_N(t)$, we have
      \begin{equation*}
            \begin{split}
                  A_FE_sG
                  =FE_sG
                  =F\mathbb{E}[G|\mathcal{A}_{s}]
                  =\mathbb{E}[FG|\mathcal{A}_{s}]
                  =E_sA_FG.
            \end{split}
      \end{equation*}

\enlargethispage{1cm}
      Thus, the first part of the lemma is proved.
      
      Let us prove the converse statement of the lemma: if for a
      given $F\in L^2(\Omega,\mathcal{A},P)$ the operator $A_F$ is
      $\mathcal{L}_{N}(t)$-measurable then $F$ is an
      $\mathcal{A}_t$-measurable function.

      Since $A_F$ is an $\mathcal{L}_{N}(t)$-measurable operator,
      we see that for any  $s\in[t,T]$
      \begin{equation*}
            A_FE_sG=E_sA_FG,\quad G\in\mathcal{H}_{N}(t),
      \end{equation*}
      or, equivalently,
      \begin{equation}\label{4.4}
            A_F\mathbb{E}[G|\mathcal{A}_{s}]=\mathbb{E}[A_FG|\mathcal{A}_{s}],
            \quad G\in\mathcal{H}_{N}(t).
      \end{equation}

      Let $s\in(t,T)$ and $(s_1,s_2]\subset(t,s]$. We take
      $$
         G:=N_{s_2}-N_{s_1} \in\mathcal{H}_{N}(t).
      $$
      Evidently, $G$ is an $\mathcal{A}_{s}$-measurable function and
      $$
         A_F\mathbb{E}[G|\mathcal{A}_{s}]=A_FG=FG,\quad
         \mathbb{E}[A_FG|\mathcal{A}_{s}]
                   =\mathbb{E}[FG|\mathcal{A}_s]=G\mathbb{E}[F|\mathcal{A}_s].
      $$
      Hence, using (\ref{4.4}), we obtain
      $$
         FG=G\mathbb{E}[F|\mathcal{A}_s].
      $$
      As a result,
      $$
         F=\mathbb{E}[F|\mathcal{A}_s],\quad s\in(t,T].
      $$
      Since the resolution of identity
      $[0,T]\ni s\mapsto E_s=\mathbb{E}[\,\cdot\,|\mathcal{A}_s]\in\mathcal{L}(\mathcal{H})$
      is a right-continuous function, the latter equality still holds
      for  $s=t$, and therefore $F$ is an
      $\mathcal{A}_t$-measurable function.
\end{proof}

As a simple consequence of Lemma \ref{l.1} we obtain the following
result.


\begin{thm}\label{t.2}
       Let $F$ belong to $L^2([0,T]\times\Omega)$.
       The family $\{A_{F}(t)\}_{t\in[0,T]}$ of the operators $A_{F}(t)$ of  multiplication by
       $F(t)=F(t,\cdot)\in L^2(\Omega,\mathcal{A},P)$ in the space $L^2(\Omega,\mathcal{A},P)$,
       $$
          L^2(\Omega,\mathcal{A},P)\supset\mathop{\rm Dom}(A_{F}(t))\ni
          G\mapsto A_{F}(t)G:=F(t)G\in L^2(\Omega,\mathcal{A},P),
       $$
       is $H$-stochastic integrable with respect to
       the normal martingale
       $N$ (i.e. belongs to $S_2$) if and only if $F$ belongs
       to the space $L_a^2([0,T]\times\Omega)$.
\end{thm}

The next theorem shows that the It\^{o} stochastic integral with
respect to the normal martingale $N$ can be interpreted as an
$H$-stochastic integral.

\begin{thm}\label{t.3}
       Let $F\in L_a^2([0,T]\times\Omega)$ and $\{A_{F}(t)\}_{t\in[0,T]}$
       be the corresponding family of the operators $A_{F}(t)$ of  multiplication
       by $F(t)$ in the space $L^2(\Omega,\mathcal{A},P)$. Then
       \begin{equation*}
             \int_{[0,T]}A_{F}(t)\,dN_t=\int_{[0,T]}F(t)\,dN_t.
       \end{equation*}
\end{thm}

\begin{proof}
        Taking into account  Theorem \ref{t.2},
        Lemma \ref{l.1} and the definitions of the integrals
        \begin{equation*}
             \int_{[0,T]}A_{F}(t)\,dN_t\quad\text{and}\quad
             \int_{[0,T]}F(t)\,dN_t,
        \end{equation*}
        it is sufficient to prove Theorem \ref{t.3} for simple functions
        $F\in L_a^2([0,T]\times\Omega)$.
        But in this case Theorem \ref{t.3} is obvious.
\end{proof}



\section{The It\^{o} integral in a Fock space as an $H$-stochastic integral}

Let us recall the definition of the It\^{o} integral in a Fock
space, see e.g. \cite{Attal} for more details. We denote by
$\mathcal{F}$ the symmetric Fock space over the real separable Hilbert
space $L^2([0,T]):=L^2([0,T],dt)$. By definition (see e.g.
\cite{BK88}),
\begin{equation*}
   {\mathcal F}:=\bigoplus_{n=0}^{\infty}{\mathcal F}_nn!,
\end{equation*}
where ${\mathcal F}_0:=\mathbb{C}$ and, for each $n\in\mathbb{N}$,
$
   {\mathcal F}_n:=(L_{{\mathbb C}}^2([0,T]))^{\mathbin{\widehat{\otimes}} n}
$ is an $n$-th symmetric tensor power $\mathbin{\widehat{\otimes}}$
of the complex Hilbert space $L_{\mathbb{C}}^2([0,T])$.  Thus, the
Fock space $\mathcal{F}$ is the complex Hilbert space of sequences
$f=(f_n)_{n=0}^{\infty}$ such that $f_n\in\mathcal{F}_n$ and
$$
   \|f\|_{{\mathcal F}}^{2}=\sum_{n=0}^{\infty}
   \|f_n\|_{{\mathcal F}_n}^{2}n!<\infty.
$$

We denote by
$
   L^2([0,T];\mathcal{F})
$ the Hilbert space of all $\mathcal{F}$-valued functions
$$
   [0,T]\ni t\mapsto f(t)\in\mathcal{F},\quad
   \|f\|_{L^2([0,T];\mathcal{F})}
   :=\Big(\int_{[0,T]}\|f(t)\|_{\mathcal{F}}^{2}\,dt\Big)^{\frac{1}{2}}<\infty
$$
with the corresponding scalar product.  A function
$f(\cdot)=(f_n(\cdot))_{n=0}^{\infty}\in L^2([0,T];\mathcal{F})$ is
called {\it It\^{o} integrable} if, for almost all $t\in[0,T]$,
$$
   f(t)=(f_0(t),\varkappa_{[0,t]}f_1(t),\ldots,\varkappa_{[0,t]^n}f_n(t),\ldots).
$$
We denote by $L_a^2([0,T];\mathcal{F})$ the set of all It\^{o} integrable functions.

Let $f$ belong to the space $L_{a,s}^2([0,T];\mathcal{F})$ of all
{\it simple It\^{o} integrable functions}. That is, $f$  belongs to
$L_a^2([0,T];\mathcal{F})$ and there exists a partition
$0=t_0<t_1<\cdots<t_n=T$ of $[0,T]$ such that
\begin{equation*}
      f(t)=\sum_{k=0}^{n-1}f_{(k)}\varkappa_{(t_k,t_{k+1}]}(t)\in\mathcal{F}
\end{equation*}
for almost all $t\in[0,T]$. The {\it It\^{o} integral} ${\mathbb
I}(f)$ of  such a function $f$ is defined by the formula
\begin{equation*}
   {\mathbb I}(f)
   :=\sum_{k=0}^{n-1}f_{(k)}\lozenge (0,\varkappa_{(t_k,t_{k+1}]},0,0,\ldots)\in\mathcal{F},
\end{equation*}
where the symbol $\lozenge$ denotes the Wick product in the  Fock
space $\mathcal{F}$. Let us recall that for given
$f=(f_n)_{n=0}^{\infty}$ and $g=(g_n)_{n=0}^{\infty}$ from
$\mathcal{F}$ the Wick product $f\lozenge g$ is defined by
\begin{equation*}
   f\lozenge g:=\Big(\sum_{m=0}^{n}f_m\mathbin{\widehat{\otimes}}g_{n-m}\Big)_{n=0}^{\infty},
\end{equation*}
provided the latter sequence belongs to the Fock space  $\mathcal{F}$.

The It\^{o} integral ${\mathbb I}(f)$  of a simple function $f\in
L_{a,s}^2([0,T];\mathcal{F})$ has the isometry property
$$
   \big\|{\mathbb I}(f)\big\|_{\mathcal{F}}^{2}
   =\int_{[0,T]}\big\|f(t)\big\|_{\mathcal{F}}^{2}\,dt,
$$
see e.g. \cite{Attal, ABT07}. Hence, extending the mapping
$$
   L_a^2([0,T];\mathcal{F})\supset L_{a,s}^2([0,T];\mathcal{F})\ni f\mapsto
   {\mathbb I}(f)\in \mathcal{F}
$$
by continuity we obtain a definition of the It\^{o} integral
${\mathbb I}(f)$ for each $f\in L_a^2([0,T];\mathcal{F})$ (we keep
the same notation ${\mathbb I}$ for the extension).

Let us show that the It\^{o} integral ${\mathbb I}(f)$  of $f\in
L_a^2([0,T];\mathcal{F})$  can be considered as an $H$-stochastic
integral. To do this  we set $
   \mathcal{H}:=\mathcal{F}
$
and consider in this space the resolution of identity
$$
   [0,T]\ni t\mapsto \mathcal{X}_tf
   :=(f_0,\varkappa_{[0,t]}f_1,\ldots,\varkappa_{[0,t]^n}f_n,\ldots)
   \in\mathcal{L}(\mathcal{F}),
   \quad f=(f_n)_{n=0}^{\infty}\in\mathcal{F}.
$$
Let $
   Z:=(0,1,0,0,\ldots)\in\mathcal{F}
$
and
$$
   [0,T]\ni t\mapsto Z_t:=\mathcal{X}_tZ=
   (0,\varkappa_{[0,t]},0,0,\ldots)\in\mathcal{F}
$$
be the corresponding  abstract martingale in the Fock space $\mathcal{F}$.
Note that now
\begin{equation*}
   \mu([0,t]):=\|Z_t\|_{\mathcal{F}}^{2}
   =\|\varkappa_{[0,t]}\|_{L_{\mathbb{C}}^2([0,T])}^{2}
   =t,\quad t\in[0,T],
\end{equation*}
i.e., $\mu$ is the Lebesgue measure on $\mathcal{B}([0,T])$.

We have the following analogues of Theorems \ref{t.2} and~\ref{t.3}. 

\enlargethispage{1.2cm}

\begin{thm}\label{t.4}
      A function $f\in L^2([0,T];\mathcal{F})$ belongs to the space  $L_{a}^2([0,T];\mathcal{F})$
      if and only if the
      corresponding operator-valued function $[0,T]\ni t\mapsto A_{f}(t)$
      whose values are operators $A_{f}(t)$ of Wick multiplication by $f(t)\in\mathcal{F}$
      in the Fock space $\mathcal{F}$,
      $$
         \mathcal{F}\supset \mathop{\rm Dom}(A_{f}(t))\ni g\mapsto A_{f}(t)g
         :=f(t)\lozenge g\in\mathcal{F},
      $$
      belongs to the space $S_2$.
\end{thm}

\begin{thm}\label{t.4.1}
       Let $f\in L_{a}^2([0,T];\mathcal{F})$ and $\{A_{f}(t)\}_{t\in[0,T]}$
       be the corresponding family of the operators $A_{f}(t)$ of Wick  multiplication
       by $f(t)\in \mathcal{F}$ in the Fock space $\mathcal{F}$. Then
       \begin{equation*}
             {\mathbb I}(f)=\int_{[0,T]}A_{f}(t)\,dZ_t.
       \end{equation*}
\end{thm}


Taking into account Theorem \ref{t.4.1}, in what follows  we will
denote the It\^{o} integral ${\mathbb I}(f)$ of $f\in
L_{a}^2([0,T];\mathcal{F})$ by $\int_{[0,T]}f(t)\,dZ_t$. Note that
this integral can be expressed in terms of the Fock space
$\mathcal{F}$. Namely, for any
$f(\cdot)=(f_n(\cdot))_{n=0}^{\infty}\in L_{a}^2([0,T];\mathcal{F})$,
we have 
      \begin{equation}\label{eq.5.1}
            \int_{[0,T]}f(t)\,dZ_t=
            (0,\hat{f}_1,\ldots,\hat{f}_n,\ldots)\in {\mathcal F},
      \end{equation}
      where, for each $n\in\mathbb{N}$ and almost all $(t_1,\ldots,t_n)\in [0,T]^n$,
      $$
        \hat{f}_n(t_1,\ldots,t_n):=\frac{1}{n}
         \sum_{k=1}^{n}f_{n-1}(t_k;t_1,\ldots,t_k\!\!\!\!\!\!\diagup\,\,,\ldots,t_n),
      $$
      i.e.,
      $\hat{f}_n$ is the symmetrization of
      $f_{n-1}(t;t_1,\ldots,t_{n-1})$ with respect to $n$ variables.



\section{A connection between the classical It\^{o} integral\\
  and the It\^{o} integral in the Fock space}

As before, let
$
   (\Omega, \mathcal{A}, P)
$ be a complete probability space with a right continuous filtration
$\{\mathcal{A}_t\}_{t \in [0,T]}$,  $\mathcal{A}_0$ be the trivial
$\sigma$-algebra containing all $P$-null sets of $\mathcal{A}$ and
$\mathcal{A}=\mathcal{A}_T$.\pagebreak

\noindent 
Let $N=\{N_t\}_{t\in[0,T]}$ be a normal martingale on $(\Omega,
\mathcal{A}, P)$ with respect to $\{\mathcal{A}_t\}_{t \in [0,T]}$,
$N_0=0$. It is known that the mapping
$$
   \mathcal{F}\ni f=(f_n)_{n=0}^{\infty}\mapsto
   If:=\sum_{n=0}^{\infty}I_n(f_n)\in
   L^2(\Omega,\mathcal{A},P)
$$
is well-defined and isometric. Here $I_0(f_0):=f_0$ and, for each
$n\in\mathbb{N}$,
$$
   I_n(f_n):=n!\int_{0}^{T}\int_{0}^{t_n}
      \cdots\Big(\int_{0}^{t_2}f_n(t_1,\ldots,t_n)\,
      dN_{t_1}\Big)\ldots \,dN_{t_{n-1}}\,dN_{t_n}
$$
is an $n$-iterated It\^{o}  integral with respect to $N$. We suppose
that the normal martingale $N$ has the {\it chaotic representation
property} (CRP). In other words, we assume that the mapping
$I:\mathcal{F}\to L^2(\Omega,\mathcal{A},P)$ is a unitary. Note that
$$
   N_t=IZ_t\in L^2(\Omega,\mathcal{A},P),
  \quad t\in[0,T],
$$
i.e., $N$ is the $I$-image of the abstract martingale $[0,T]\ni
t\mapsto Z_t= (0,\varkappa_{[0,t]},0,0,\ldots)\in\mathcal{F}.$

The Brownian motion, the compensated Poisson process and some
Az${\rm \acute{e}}$ma martingales are examples of normal martingales
which possess the CRP, see e.g. \cite{Emery89, MPM98}.

We note that the spaces $L^2([0,T]\times\Omega)$ and
$L^2([0,T];\mathcal{F})$ can be understood as the tensor products
$L^2([0,T])\otimes L^2(\Omega,\mathcal{A},P)$ and $L^2([0,T])\otimes
\mathcal{F}$, respectively. Therefore,
$$
   1\otimes I:L^2([0,T];\mathcal{F})\to L^2([0,T]\times\Omega)
$$
is a unitary operator.

\enlargethispage{1cm}
The next result gives a relationship between the classical It\^{o}
integral with respect to the normal martingale with CRP and the
It\^{o} integral in the Fock space $\mathcal{F}$.

\begin{thm}\label{t.7.1}
      We have
      $$
         L_a^2([0,T]\times\Omega)=(1\otimes I)L_a^2([0,T];\mathcal{F})
      $$
      and, for arbitrary $f\in L_{a}^2([0,T];\mathcal{F})$,
\begin{equation*}
      I\Big(\int_{[0,T]}f(t)\,dZ_t\Big)
      =\int_{[0,T]}If(t)\,dN_t.
\end{equation*}
\end{thm}

Since $N$ has CRP, for any $F\in L_a^2([0,T]\times\Omega)$ there
exists a uniquely defined vector
$f(\cdot)=(f_n(\cdot))_{n=0}^{\infty}\in L_{a}^2([0,T];\mathcal{F})$
such that
      $$
         F(t)=If(t)=\sum_{n=0}^{\infty}I_n(f_n(t))
      $$
for almost all $t\in[0,T]$. Hence, using Theorem \ref{t.7.1} and
equality (\ref{eq.5.1}) we obtain
      \begin{equation*}
            \int_{[0,T]}F(t)\,dN_t=I\Big(\int_{[0,T]}f(t)\,dZ_t\Big)
            =\sum_{n=1}^{\infty}I_n(\hat{f}_n)\in L^2(\Omega, \mathcal{A}, P).
      \end{equation*}

It should be noticed that the right hand side of the latter equality
was used by Hitsuda~\cite{Hitsuda} and Skorohod~\cite{Skorohod} to
define an extension of the It\^{o} integral. Namely, a function
      $$
         F(\cdot)=\sum_{n=0}^{\infty}I_n(f_n(\cdot))\in L^2([0,T]\times\Omega)
      $$
is  Hitsuda-Skorohod integrable if and only if
$$
   \sum_{n=1}^{\infty}I_n(\hat{f}_n)\in L^2(\Omega, \mathcal{A}, P)
   \quad \text{or, equivalently,}\quad
   \sum_{n=1}^{\infty}\|\hat{f}_n\|_{\mathcal F_n}^{2}n!<\infty.
$$
The corresponding {\it Hitsuda-Skorohod integral}  ${\mathbb
I}_{\mathop{\rm {HS}}}(F)$ of $F$ is defined by the formula
      \begin{equation*}
            {\mathbb I}_{\mathop{\rm {HS}}}(F):=\sum_{n=1}^{\infty}I_n(\hat{f}_n).
      \end{equation*}


{\it{Acknowledgments.}} I am grateful to Prof. Yu.~M.~Berezansky
for the statement of a question and helpful discussions and the
referee for the useful remarks. This work has been partially
supported by the DFG, Project 436 UKR 113/78/0-1 and by the
Scientific Program of National Academy of Sciences of Ukraine,
Project No~0107U002333.

\end{document}